\documentclass[reqno]{amsart} 
\usepackage{amssymb,amscd,url}
\usepackage{tikz-cd}

\begin{document}

\allowdisplaybreaks


\title{Upper Bounds on the Sizes of Finite Orbits for Unramified Morphisms}
\author[Young Kyun Kim]{Young Kyun Kim}

\address{Department of Mathematical Sciences, Seoul National University, Gwanak-ro 1, Gwanak-gu, Seoul 08826, Republic of Korea}
\email{young98kim@snu.ac.kr}
\subjclass[2020]{Primary: 37P55; Secondary:  37P05, 37P20, 37P15 }
\keywords{algebraic dynamics, finite orbits, unramified morphism}





\newtheorem{theorem}{Theorem}
\newtheorem{lemma}[theorem]{Lemma}
\newtheorem{conjecture}[theorem]{Conjecture}
\newtheorem{proposition}[theorem]{Proposition}
\newtheorem{corollary}[theorem]{Corollary}
\newtheorem{claim}[theorem]{Claim}

\newtheorem{definition}[theorem]{Definition}
\newtheorem{example}[theorem]{Example}
\newtheorem{remark}[theorem]{Remark}
\newtheorem{question}[theorem]{Question}

\theoremstyle{remark}
\newtheorem*{acknowledgement}{Acknowledgements}


\newenvironment{notation}[0]{%
  \begin{list}%
    {}%
    {\setlength{\itemindent}{0pt}
     \setlength{\labelwidth}{4\parindent}
     \setlength{\labelsep}{\parindent}
     \setlength{\leftmargin}{5\parindent}
     \setlength{\itemsep}{0pt}
     }%
   }%
  {\end{list}}

\newenvironment{parts}[0]{%
  \begin{list}{}%
    {\setlength{\itemindent}{0pt}
     \setlength{\labelwidth}{1.5\parindent}
     \setlength{\labelsep}{.5\parindent}
     \setlength{\leftmargin}{2\parindent}
     \setlength{\itemsep}{0pt}
     }%
   }%
  {\end{list}}
\newcommand{\Part}[1]{\item[\upshape#1]}

\def\Case#1#2{%
\paragraph{\textbf{\boldmath Case #1: #2.}}\hfil\break\ignorespaces}

\renewcommand{\a}{\alpha}
\renewcommand{\b}{\beta}
\newcommand{\g}{\gamma}
\renewcommand{\d}{\delta}
\newcommand{\e}{\epsilon}
\newcommand{\f}{\varphi}
\newcommand{\bfphi}{{\boldsymbol{\f}}}
\renewcommand{\l}{\lambda}
\renewcommand{\k}{\kappa}
\newcommand{\lhat}{\hat\lambda}
\newcommand{\m}{\mu}
\newcommand{\bfmu}{{\boldsymbol{\mu}}}
\renewcommand{\o}{\omega}
\renewcommand{\r}{\rho}
\newcommand{\rbar}{{\bar\rho}}
\newcommand{\s}{\sigma}
\newcommand{\sbar}{{\bar\sigma}}
\renewcommand{\t}{\tau}
\newcommand{\z}{\zeta}

\newcommand{\D}{\Delta}
\newcommand{\G}{\Gamma}
\newcommand{\F}{\Phi}
\renewcommand{\L}{\Lambda}

\newcommand{\ga}{{\mathfrak{a}}}
\newcommand{\gb}{{\mathfrak{b}}}
\newcommand{\gm}{{\mathfrak{m}}}
\newcommand{\gn}{{\mathfrak{n}}}
\newcommand{\gp}{{\mathfrak{p}}}
\newcommand{\gP}{{\mathfrak{P}}}
\newcommand{\gq}{{\mathfrak{q}}}

\newcommand{\Abar}{{\bar A}}
\newcommand{\Ebar}{{\bar E}}
\newcommand{\kbar}{{\bar k}}
\newcommand{\Kbar}{{\bar K}}
\newcommand{\Pbar}{{\bar P}}
\newcommand{\Sbar}{{\bar S}}
\newcommand{\Tbar}{{\bar T}}
\newcommand{\gbar}{{\bar\gamma}}
\newcommand{\lbar}{{\bar\lambda}}
\newcommand{\ybar}{{\bar y}}
\newcommand{\phibar}{{\bar\f}}

\newcommand{\Acal}{{\mathcal A}}
\newcommand{\Bcal}{{\mathcal B}}
\newcommand{\Ccal}{{\mathcal C}}
\newcommand{\Dcal}{{\mathcal D}}
\newcommand{\Ecal}{{\mathcal E}}
\newcommand{\Fcal}{{\mathcal F}}
\newcommand{\Gcal}{{\mathcal G}}
\newcommand{\Hcal}{{\mathcal H}}
\newcommand{\Ical}{{\mathcal I}}
\newcommand{\Jcal}{{\mathcal J}}
\newcommand{\Kcal}{{\mathcal K}}
\newcommand{\Lcal}{{\mathcal L}}
\newcommand{\Mcal}{{\mathcal M}}
\newcommand{\Ncal}{{\mathcal N}}
\newcommand{\Ocal}{{\mathcal O}}
\newcommand{\Pcal}{{\mathcal P}}
\newcommand{\Qcal}{{\mathcal Q}}
\newcommand{\Rcal}{{\mathcal R}}
\newcommand{\Scal}{{\mathcal S}}
\newcommand{\Tcal}{{\mathcal T}}
\newcommand{\Ucal}{{\mathcal U}}
\newcommand{\Vcal}{{\mathcal V}}
\newcommand{\Wcal}{{\mathcal W}}
\newcommand{\Xcal}{{\mathcal X}}
\newcommand{\Ycal}{{\mathcal Y}}
\newcommand{\Zcal}{{\mathcal Z}}

\renewcommand{\AA}{\mathbb{A}}
\newcommand{\BB}{\mathbb{B}}
\newcommand{\CC}{\mathbb{C}}
\newcommand{\FF}{\mathbb{F}}
\newcommand{\GG}{\mathbb{G}}
\newcommand{\NN}{\mathbb{N}}
\newcommand{\PP}{\mathbb{P}}
\newcommand{\QQ}{\mathbb{Q}}
\newcommand{\RR}{\mathbb{R}}
\newcommand{\ZZ}{\mathbb{Z}}

\newcommand{\bfa}{{\mathbf a}}
\newcommand{\bfb}{{\mathbf b}}
\newcommand{\bfc}{{\mathbf c}}
\newcommand{\bfd}{{\mathbf d}}
\newcommand{\bfe}{{\mathbf e}}
\newcommand{\bff}{{\mathbf f}}
\newcommand{\bfg}{{\mathbf g}}
\newcommand{\bfp}{{\mathbf p}}
\newcommand{\bfr}{{\mathbf r}}
\newcommand{\bfs}{{\mathbf s}}
\newcommand{\bft}{{\mathbf t}}
\newcommand{\bfu}{{\mathbf u}}
\newcommand{\bfv}{{\mathbf v}}
\newcommand{\bfw}{{\mathbf w}}
\newcommand{\bfx}{{\mathbf x}}
\newcommand{\bfy}{{\mathbf y}}
\newcommand{\bfz}{{\mathbf z}}
\newcommand{\bfA}{{\mathbf A}}
\newcommand{\bfF}{{\mathbf F}}
\newcommand{\bfB}{{\mathbf B}}
\newcommand{\bfD}{{\mathbf D}}
\newcommand{\bfG}{{\mathbf G}}
\newcommand{\bfI}{{\mathbf I}}
\newcommand{\bfM}{{\mathbf M}}
\newcommand{\bfzero}{{\boldsymbol{0}}}

\newcommand{\Aut}{\operatorname{Aut}}
\newcommand{\Base}{\mathcal{B}} 
\newcommand{\CM}{\operatorname{CM}}   
\newcommand{\codim}{\operatorname{codim}}
\newcommand{\diag}{\operatorname{diag}}
\newcommand{\Disc}{\operatorname{Disc}}
\newcommand{\Div}{\operatorname{Div}}
\newcommand{\Dom}{\operatorname{Dom}}
\newcommand{\Ell}{\operatorname{Ell}}   
\newcommand{\End}{\operatorname{End}}
\newcommand{\Fbar}{{\bar{F}}}
\newcommand{\Gal}{\operatorname{Gal}}
\newcommand{\GL}{\operatorname{GL}}
\newcommand{\Hom}{\operatorname{Hom}}
\newcommand{\hplus}{h^{\scriptscriptstyle+}}
\newcommand{\Index}{\operatorname{Index}}
\newcommand{\Image}{\operatorname{Image}}
\newcommand{\liftable}{{\textup{liftable}}}
\newcommand{\hhat}{{\hat h}}
\newcommand{\Ker}{{\operatorname{ker}}}
\newcommand{\Lift}{\operatorname{Lift}}
\newcommand{\Mat}{\operatorname{Mat}}
\newcommand{\MOD}[1]{~(\textup{mod}~#1)}
\newcommand{\Mor}{\operatorname{Mor}}
\newcommand{\Moduli}{\mathcal{M}}
\newcommand{\Norm}{{\operatorname{\mathsf{N}}}}
\newcommand{\notdivide}{\nmid}
\newcommand{\normalsubgroup}{\triangleleft}
\newcommand{\NotDom}{\operatorname{NotDom}}
\newcommand{\NS}{\operatorname{NS}}
\newcommand{\odd}{{\operatorname{odd}}}
\newcommand{\onto}{\twoheadrightarrow}
\newcommand{\ord}{\operatorname{ord}}
\newcommand{\Orbit}{\mathcal{O}}
\newcommand{\Per}{\operatorname{Per}}
\newcommand{\Perp}{\operatorname{Perp}}
\newcommand{\PrePer}{\operatorname{PrePer}}
\newcommand{\PGL}{\operatorname{PGL}}
\newcommand{\Pic}{\operatorname{Pic}}
\newcommand{\Prob}{\operatorname{Prob}}
\newcommand{\Qbar}{{\bar{\QQ}}}
\newcommand{\rank}{\operatorname{rank}}
\newcommand{\Rat}{\operatorname{Rat}}
\newcommand{\rel}{{\textup{rel}}}  
\newcommand{\Resultant}{\operatorname{Res}}
\renewcommand{\setminus}{\smallsetminus}
\newcommand{\shat}{{\hat s}}
\newcommand{\sing}{{\textup{sing}}} 
\newcommand{\sgn}{\operatorname{sgn}} 
\newcommand{\SL}{\operatorname{SL}}
\newcommand{\Span}{\operatorname{Span}}
\newcommand{\Spec}{\operatorname{Spec}}
\newcommand{\Support}{\operatorname{Supp}}
\newcommand{\That}{{\hat T}}  
\newcommand{\tors}{{\textup{tors}}}
\newcommand{\Trace}{\operatorname{Trace}}
\newcommand{\tr}{{\textup{tr}}} 
\newcommand{\UHP}{{\mathfrak{h}}}    
\newcommand{\<}{\langle}
\renewcommand{\>}{\rangle}

\newcommand{\ds}{\displaystyle}
\newcommand{\longhookrightarrow}{\lhook\joinrel\longrightarrow}
\newcommand{\longonto}{\relbar\joinrel\twoheadrightarrow}

\newcount\ccount \ccount=0
\def\cc{\global\advance\ccount by1{c_{\the\ccount}}}


\begin{abstract}
We prove that for a dynamical system on an algebraic variety over $\overline{\mathbb{Q}}$ generated by finitely many unramified endomorphisms, it is decidable whether a given point has a finite orbit.
This is achieved by establishing an effective upper bound for the size of finite orbits in integral algebraic dynamical systems with unramified morphisms.
\end{abstract}


\maketitle


\section{Introduction}
One of the central questions in arithmetic dynamics concerns bounding the sizes of preperiodic orbits (see \cite{Benedetto, Fakhruddin, Looper, Morton, Poonen, Silverman}).
In this paper, we obtain upper bounds for the sizes of preperiodic orbits and establish the decidability of preperiodic points in arithmetic dynamical systems generated by unramified endomorphisms.
We begin by introducing basic definitions in dynamical systems. 
Given a set $X$ with a collection of endomorphisms $S\subseteq \End(X)$, we shall call the pair $(X, S)$ a \emph{dynamical system}.
For $x\in X$, we shall say that $x$ is an \emph{$S$-finite} point if its forward $S$-orbit
\[
  \Orbit_S(x)=\{ f(x) : f\in \langle S \rangle \}
\]
is a finite set, where $\langle S \rangle$ is the monoid generated by $S$ under composition. Furthermore, if $S$ acts by permutation on $\Orbit_S(x)$, then $x$ is called an \emph{$S$-periodic} point.

In the case of a dynamical system $(X,f)$ with a single map $f$, a point $x\in X$ is $f$-periodic if and only if $f^n(x)=x$ for some $n\geq 1$.
Also, an $f$-finite point is equivalent to a preperiodic point in the usual sense, that is, $f^n(x)=f^{n+m}(x)$ for some $n\geq 0$ and $m>0$.
When considering a dynamical system with a single map, we shall use the term \emph{$f$-preperiodic} to refer to $f$-finite points (see Remark~\ref{remark:defofpreper}).

\subsection{Decidability of points with finite orbit}
A natural question is to determine the preperiodicity of a given point under a given self-map.
For systems of maps, Whang \cite{Whang} proved that the periodic points of algebraic varieties (reduced separated schemes of finite type over a field) over $\Qbar$ are decidable and posed the following question:
\begin{question}
\label{question:decide}
Given a finite set $S$ of endomorphisms of an algebraic variety $X/\Qbar$, is the subset of points with finite $S$-orbit decidable?
\end{question}
We provide an affirmative answer to Question~\ref{question:decide} for dynamical systems with unramified endomorphisms.
\begin{theorem}
\label{decidable}
Let $S$ be a finite set of unramified endomorphisms of an algebraic variety $X/\Qbar$.
Given $x\in X(\Qbar)$, it is decidable whether
or not $x$ is $S$-finite.
\end{theorem}

One approach to solving Question~\ref{question:decide} involves obtaining effective upper bounds on the size of finite orbits.
Given a dynamical system $(X, S)$, suppose there exists an effective constant $C$ such that $|\Orbit_S(x)|\leq C$, for all $S$-finite points $x\in X$.
Consider the ascending chain of finite sets $O_1 \subseteq O_2 \subseteq \cdots \subseteq O_n \subseteq \cdots$ where
$O_1=\{x\}$ and $O_{n+1}=O_n \cup \{ f(y) : y\in O_n, \ f\in S \}$.
If $|O_{\lceil C \rceil}|\leq C$, then $x$ is $S$-finite, otherwise if $|O_{\lceil C \rceil}|>C$ then $x$ is not $S$-finite, where ${\lceil C \rceil}$ is smallest integer greater than or equal to $C$.
Thus, we can decide whether a given point $x\in X$ is $S$-finite point or not.

In line with the description above, Theorem~\ref{decidable} is accomplished by finding effective upper bounds for the finite orbits of an integral model of the given arithmetical dynamical system.
The proof of Theorem~\ref{decidable} will be presented in Section~\ref{section:decide}.

\subsection{Upper bounds on the size of finite orbits}

Before presenting our main results, we provide intuition for formally unramified morphisms by beginning with a simple application of one of our main theorems. The following is a particular case of Theorem~\ref{theorem:manymap}.

\begin{corollary}
\label{corollary:special case}
Let $n$ and $N$ be positive integers and let $p$ be a prime number such that $p\nmid N$.
Let $S$ be a finite set of polynomial maps
\[
f=(f_1, \ldots, f_n):\ZZ[1/N]^n\to\ZZ[1/N]^n
\]
where each $f_i$ is of the form
\[
f_i(x_1, \ldots, x_n)=x_i + p\cdot g_i(x_1, \ldots, x_n) + h_i(x_1^p, \ldots, x_n^p)
\]
for some $g_i, h_i \in\ZZ[1/N][x_1,\ldots,x_n]$.
Then there exists an effective constant $C=C(n,N)$ such that,
\[
|\Orbit_S(x)|\leq (C p^n |S| +1)^{C p^n -1}
\]
for all $S$-finite points $x\in(\ZZ[1/N])^n$.
\end{corollary}

Note that for the maps Corollary~\ref{corollary:special case}, the Jacobian matrix reduces to the identity matrix modulo $p$. This is a special case of the general notion of an unramified morphism between schemes. We recall the definition of formally unramified morphisms.

\begin{definition}
Let $f:X\to Y$ be a scheme morphism.
For a point $x\in X$, we shall say $f$ is \emph{formally unramified at $x$}, if the stalk at $x$ of the relative sheaf of differentials is zero, that is $\Omega_{X/Y,x}=0$.
We say that $f$ is \emph{formally unramified} if the relative sheaf of differentials vanishes, that is $\Omega_{X/Y}=0$.
\end{definition}

In the case where $X$ and $Y$ are schemes of finite type over a field, $f$ being unramified at a point $x\in X$ implies that the differential map $df|_x : T_{X,x} \to T_{Y,y}\otimes_{\kappa(y)}\kappa(x)$ between the Zariski tangent spaces is injective. In the setting of Corollary~\ref{corollary:special case}, the reduction of $df$ modulo $p$ is the identity, so this injectivity holds.

\begin{definition}
Let $f:X\to X$ be a scheme endomorphism over a field $k$. Then we shall say $f$ has \emph{only unramified periodic $k$-points}, if $f$ is formally unramified at every periodic point $x\in X(k)$.
\end{definition}

\begin{theorem}
\label{theorem:singlemap}
Let $X$ be a separated scheme of finite presentation over a finitely presented domain $R$ of characteristic zero.
Let $\gm\leq R$ be a maximal ideal.
Let $f:X\to X$ be an endomorphism such that $f_{R/\gm}$ has only unramified periodic $R/\gm$-points.
There exists an effective constant $C=C(X,R)$ such that,
\[
|\Orbit_f(x)| < |X(R/\gm)|+C
\]
for all $f$-preperiodic $x\in X(R)$. Here $f_{R/\gm}$ is the base change of $f$ over $\Spec(R/\gm)$.
\end{theorem}

In certain cases, we obtain sharper bounds than the one given in Theorem~\ref{theorem:singlemap}; see Theorem~\ref{theorem:singlemap2} and Example~\ref{example:ZZ}. In proving Theorem~\ref{theorem:singlemap}, we bound both the preperiodic part and the periodic parts of the orbit.
In particular, for an $f$-preperiodic point $x\in X(R)$, let the smallest integers $n\geq 0$ and $m>0$ satisfying $f^n(x)=f^{n+m}(x)$.
Then $|\Orbit_f(x)|=n+m$, and we bound $n$ and $m$ by $|X(R/\gm)|$ and $C$, respectively.
The constant $C$ is the upper bound on the size of periodic orbits proven by Whang \cite{Whang}.

We now extend our analysis to dynamical systems with finitely many maps,
where we establish an upper bound on the size of finite orbits under systems of unramified endomorphisms:

\begin{theorem}
\label{theorem:manymap}
Let $X$ be a separated scheme of finite presentation over a finitely presented domain $R$ of characteristic zero.
Let $\gm\leq R$ be a maximal ideal.
There exists an effective constant $C=C(X,R)$ such that
for any finite subset
\[
S\subseteq \{f\in \End(X/R) : f_{R/\gm} \text{ is unramified}\},
\]
we have
\[
|\Orbit_S(x)|\leq (C|X(R/\gm)||S|+1)^{C|X(R/\gm)|-1}
\]
for all $S$-finite points $x\in X(R)$.
\end{theorem}

For a proper scheme $X$ over a Dedekind domain $R$, we have $X(R)=X(F)$ where $F$ is the field of fractions of $R$. In this case, our theorems provides a bound on the orbit sizes of rational preperiodic points for unramfied morphisms.
The key to proving Theorem~\ref{theorem:singlemap} and Theorem~\ref{theorem:manymap} lies in Proposition~\ref{proposition:key_fixed}.
Our strategy is to infinitesimally lift the fixed points modulo $\gm$ to fixed points modulo $\gm^n$ for any $n$.
This can be done since formally unramified maps satisfy a unique infinitesimal lifting property(see Proposition~\ref{proposition:lifting}).
Specifically, if a map modulo $\gm$ is formally unramified at a fixed point, then the fixed points can be uniquely lifted to a first order thickening. 
Thus, repeated lifting of unramified fixed points modulo $\gm$ ensures that they remain fixed points modulo $\gm^n$ for any $n$, and consequently, they are fixed points in the completion of $R$ by $\gm$. 
Theorem~\ref{theorem:singlemap} will be proved in Section~\ref{section:bound1} and the proof of Theorem~\ref{theorem:manymap} will be presented in Section~\ref{section:manymap}.

In arithmetic dynamical systems, there has been considerable study on upper bounds for the size of periodic and preperiodic orbits.
Existing results on non-proper schemes primarily focus on bounding the size of periodic orbits.
For instance, Pezda \cite{PezdaZ2, Pezda3, PezdaZN} obtained upper bounds for the size of periodic orbits of a single polynomial endomorphism on $\ZZ^n$. 
Whang \cite{Whang} further proved a uniform upper bound on the size of periodic orbits for finitely presented separated schemes over a finitely presented ring of characteristic zero under any set of endomorphisms. For varieties over a discrete valuation ring, Fakhruddin \cite{Fakhruddin} showed the finiteness on the size of periodic orbits of a single endomorphism.

In the case of projective spaces, Northcott's groundbreaking result \cite{Northcott} proved the finiteness of preperiodic points of projective space endomorphisms of degree greater than $1$ over a number field. Subsequently, Morton and Silverman \cite{Morton} conjectured the existence of a uniform upper bound on the number of preperiodic points for endomorphisms of projective space over number fields.
This conjecture, known as the Morton-Silverman Uniform Boundedness Conjecture, has been the subject of extensive study; see \cite{Benedetto, Fakhruddin, Looper, Morton, Poonen, Silverman}.

In this context, a natural question arises: does there exist an upper bound on the size of preperiodic orbits of non-proper schemes, such as the affine space?
\begin{question}
Given a polynomial map $f:\AA^N \to\AA^N$, does there exist a constant $C=C(f)$ such that
\[
  |\Orbit_S(x)|<C
\]
for all $f$-preperiodic points $x$?
\end{question}
However, the following example from \cite{Bell} shows that such upper bound does not exist.
\begin{example}
\label{example:counter}
Let $\{F_n\}$ be the Fibonacci sequence with $F_0=0$, $F_1=1$.
Consider the following map
\begin{align*}
    f:\AA^2 &\to \AA^2\\
    (x,y)&\mapsto (xy, x^2+xy).
\end{align*}
Let $a_n=(F_n, -F_{n-1})$. Then $f(a_n)=(-F_nF_{n-1}, F_nF_{n-2})=-F_na_{n-1}$
and $f^2(a_1)=(0,0).$ Hence
\[
  |\Orbit_f(a_n)|=n+2.
\]
Also, note that $f^{-1}(a_n)\cap\AA^2(\ZZ)=\emptyset$.
\end{example}

The Jacobian determinant of the map given in Example~\ref{example:counter} is zero at the fixed point $(0,0)$, demonstrating that our condition of unramifiedness is necessary. 

\begin{remark}
\label{remark:defofpreper}
Let $(X,f)$ be dynamical system with a single map and let $\Orbit_f(x)$ be a finite orbit. For some $n\geq0$, $f^n(\Orbit_f(x))$ is a periodic orbit. However, in a dynamical system with more than one map, an $S$-finite orbit $\Orbit_S(x)$ may not contain any $S$-periodic orbits. For example, consider the dynamical system $(\ZZ, S)$ with $S=\{x^2-3x, x^2-3x+3\}$. In this case, no subset of $\Orbit_S(0)=\{0,3\}$ is $S$-periodic. This explains our use of the term \emph{$S$-finite} rather than ``pre-periodic" for points with finite orbits in general multi-map dynamical systems.
\end{remark}

\subsection{Acknowledgements}
I would like to thank Junho Peter Whang for valuable discussions and comments. I also thank Jae-Hwan Choi for help with editing this paper. This work was supported by the Samsung Science and Technology Foundation under Project Number SSTF-BA2201-03.

\section{Notations and Preliminaries}
\label{section:preliminary}

In this section we review basic notation and recall important properties of formally unramified morphisms.
For a dynamical system $(X,S)$ and a self-map $f\in S$, we will set the following notations.
\begin{notation}
\item[$f^n$]
the $n$th iteration of $f$, i.e. $f^n=\underbrace{{f\circ \cdots \circ f}}_\text{n times}$
\item[$\Per(X, S)$]
the subset of $X$ consisting of all $S$-periodic points.
\item[$\mathrm{Fin}(X,S)$]
the subset of $X$ consisting of all $S$-finite points.
\end{notation}
For a scheme morphism $f:X\to X$ over a ring $R$ and ring homomorphism $R\to T$, we will set the following notations.
\begin{notation}
\item[$f_{T}$]
the base change of $f$ over $\Spec(T)$.
\item[$\Omega_f$]
the relative sheaf of differentials of $f$.
\item[$X^f$]
the fixed locus of $f$.
\end{notation}

\begin{definition}
Given a dynamical system $(X, S)$. 
\begin{parts}
\Part{(a)} For $x\in X$, we say $x$ is \emph{$S$-fixed point}, if the action
of $S$ on $x$ is trivial. 
\Part{(b)} Let $f\in S$. We say $x$ is \emph{eventually $f$-fixed point},
if $f^n(x)=f^{n+1}(x)$, for some $n\geq 0$.
\end{parts}
\end{definition}

\begin{definition}
Let $f:X\to X$ be an endomorphism of a scheme over a field $k$. We shall say $f$ \emph{has only unramified fixed $k$-points}, if $f$ is formally unramified at every fixed point $x\in X^f(k)$.
\end{definition}

\begin{definition}
A morphism of schemes $f:X\to Y$ is called \emph{unramified}, if $f$ is locally of finite type and formally unramified.
\end{definition}

We recall the infinitesimal lifting property of formally unramified morphisms.
\begin{proposition}
\label{proposition:lifting}
Let $f:X\to Y$ be a morphism of schemes. For $x\in X$, the following are equivalent.
\begin{parts}
\Part{(i)} $\Omega_{f,x}=0$,
\Part{(ii)} For any $\Ocal_{Y, f(x)}$-algebra $A$ and any ideal $I\leq A$ with $I^2=0$ and for any commuting square of $\Ocal_{Y,f(x)}$-algebras
\[\begin{tikzcd}
    A/I & \Ocal_{X,x} \arrow[l] \arrow[ld, dashed] \\
    A \arrow[u] & \Ocal_{Y,f(x)} \arrow[u, swap] \arrow[l]
\end{tikzcd}\]
the diagonal map is unique.
\end{parts}
\end{proposition}
\begin{proof}
Given an $\Ocal_{Y, f(x)}$-algebra $A$ and any ideal $I\leq A$ with $I^2=0$, suppose we have a commuting square, 
\[\begin{tikzcd}
    A/I & \Ocal_{X,x} \arrow[l] \\
    A \arrow[u] & \Ocal_{Y,f(x)} \arrow[u, swap] \arrow[l]
\end{tikzcd}.\]
Let $g:\Ocal_{X,x}\to A$ and $g':\Ocal_{X,x}\to A$ be homomorphisms that fit in the diagonal of the above diagram. Then $\theta=g-g'$ is a $\Ocal_{Y,f(x)}$-derivation of $\Ocal_{X,x}$ into $I$ (see \cite[Exer II.8.6(a)]{Hartshorne}). Since $\mathrm{Der}_{\Ocal_{Y,f(x)}}(\Ocal_{X,x}, I)=\Hom_{\Ocal_{X, x}}(\Omega_{\Ocal_{X, x}/\Ocal_{Y, f(x)}}, I)=\Hom_{\Ocal_{X, x}}(\Omega_{f,x}, I)=0$, we have $g=g'$.\\
Suppose the converse. Let $J=\Ker(\Ocal_{X, x}\otimes_{\Ocal_{Y, f(x)}}\Ocal_{X, x}\to \Ocal_{X, x})$ and let $A=(\Ocal_{X, x}\otimes_{\Ocal_{Y, f(x)}}\Ocal_{X, x})/J^2$ and $I=J/J^2$. Note that $J/J^2=\Omega_{f,x}$. Then consider the maps $g:\Ocal_{X,x}\to A$ defined by $z \mapsto z\otimes 1 \bmod J^2$ and $g':\Ocal_{X,x}\to A$ defined by $z\mapsto 1\otimes z \bmod J^2$. Then by our hypothesis about the diagonal map being unique, we have $g=g'$. Note that $g-g'=0$ is the universal derivation, thus, $\Omega_{f,x}=0$.
\end{proof}

\begin{lemma}[Spreading out]
\label{lemma:spreading}
Let $R$ be a domain and let $K=\mathrm{Frac}(R)$ be its field of fractions.
Let $f:\Xcal \to \Xcal'$ be a morphism between schemes of finite type over $R$.
If $f_K : \Xcal_K \to \Xcal'_K$ is unramified then there exists $\ell \in R$ such that $f_{R[1/\ell]} : \Xcal_{R[1/\ell]}\to \Xcal'_{R[1/\ell]}$ is unramified.
\end{lemma}
\begin{proof}
Let $V_i=\Spec A_i\subseteq \Xcal'$ be an affine open subset and let $U_{ij}=\Spec B_{ij}\subseteq f^{-1}(V_i)$ be an affine open subset.
Denote $f_{ij}:=f|_{U_{ij}}:U_{ij}\to V_i$.
Put $A_i= R[y_1, \ldots, y_m]/J_i$ and $B_{ij}=R[x_1, \ldots, x_n]/I_{ij}$ and let $f^{\#}_{ij}:A_i\to B_{ij}$ be the corresponding $R$-algebra homomorphism of $f_{ij}$.
Then we have an exact sequence of modules of differentials,
\[
\Omega_{A_i/R}\otimes_{A_i}B_{ij} \to \Omega_{B_{ij}/R} \to \Omega_{B_{ij}/A_i} \to 0.
\]
Since $f$ is unramified over $K$, tensor the above sequence to obtain the following
\[
\begin{tikzcd}[row sep=0.8em]
\Omega_{A_i/R}\otimes_{A_i}B_{ij} \otimes_R K \arrow[r, "D_{ij}"] & \Omega_{B_{ij}/R} \otimes_R K \arrow[r] & \Omega_{B_{ij}\otimes_R K/A_i \otimes_R K} \arrow[equal]{d} \arrow[r] & 0\\
& & 0 &
\end{tikzcd}.
\]
Thus, we have a surjective morphism of $K$-modules,
\[
\begin{tikzcd} 
\Omega_{(K[y_1, \ldots, y_m]/J_{i})/K} \otimes_{K[y_1,\ldots,y_m]/J_i} K[x_1, \ldots, x_n]/I_{ij} \arrow[r, tail, twoheadrightarrow, "D_{ij}"] & \Omega_{(K[x_1, \ldots, x_n]/I_{ij})/K}.
\end{tikzcd}
\]
For each generator
$dx_\alpha\in\Omega_{(K[x_1, \cdots, x_n]/I_{ij})/K}$ there are
$g_{\alpha\beta}\in K[x_1, \cdots, x_n]/I_{ij}$ such that
\[
D_{ij}\left( \sum_{\beta=1}^{m} g_{\alpha\beta}dy_\beta \right)=dx_\alpha.
\]
Expanding the above we have 
\begin{equation}
\label{jac}
    \sum_{\gamma=1}^{n}\sum_{\beta=1}^{m} g_{\alpha\beta}\partial_\gamma f_\beta dx_{\gamma} =dx_\alpha.
\end{equation}
where $f_\beta=f^{\#}_{ij}(y_\beta)$.
By chasing denominators of the $g_{\alpha\beta}$, let $\ell_{ij}\in R$ such that $g_{\alpha\beta}$ lies in $R[1/\ell]$.
Thus the following
\begin{equation}
\label{jac2}
   D_{ij}\left( \sum_{\beta=1}^{m} \ell_{ij}g_{\alpha\beta}dy_\beta \right)=\sum_{\gamma=1}^{n}\sum_{\beta=1}^{m} \ell_{ij}g_{\alpha\beta}\partial_\gamma f_\beta dx_{\gamma} =\ell_{ij}dx_\alpha.
\end{equation}
is an equation in $\Omega_{(R[x_1, \cdots, x_n]/I_{ij})/R}$ for all $\alpha$.
Let $\ell=\prod_{i,j} \ell_{ij}$ and put $U=\Spec R[\ell^{-1}]$.
Then \eqref{jac2} implies that the first arrow of the following exact sequence is surjective,
\[
f^*_U \Omega_{\Xcal'_U}\to \Omega_{\Xcal_U}\to \Omega_{f_U}\to 0,
\]
implying $\Omega_{f_U}=0$ and therefore, $f_U$ is unramified.
\end{proof}

\begin{remark}
In the proof of Lemma~\ref{lemma:spreading}, if the generators of $J_i$ and $I_{ij}$ of $A_i$ and $B_{ij}$ and the map $f^\#_{ij}:A_i\to B_{ij}$ are explicitly given then we can solve the linear equation of \eqref{jac} and effectively compute the $g_{\alpha\beta}$.
Hence we can effectively obtain the unramified morphism $f_{R[1/\ell]}$.
\end{remark}

\begin{lemma}
\label{lemma:basechange}
Let $f:X\to Y$ be a morphism of schemes of finite type over a Noetherian ring $R$.
Let $\gm\leq R$ be a maximal ideal and let $\bar{X}=X_{R/\gm}$.
Let $\bar{x}$ and $x$ be points of $\bar{X}$ and $X$ respectively such that $\bar{x}$ maps to $x$ by the embedding $i: \bar{X} \to X$.
Let $\bar{f}$ be the restriction of $f$ to $\bar{X}$.
If $\Omega_{\bar{f}, \bar{x}}=0$ then $\Omega_{f,x}=0$.
\end{lemma}
\begin{proof}
Since $Y$ is Noetherian and $f$ is of finite type, note that $\Omega_{f}$ is a coherent $O_X$-module.
\begin{align*}
    \Omega_{f,x}\otimes_{R_\gm}R_\gm/\gm R_\gm &\cong \Omega_{f,x} \otimes_{\Ocal_{X,x}}\Ocal_{X,x} \otimes_{R_\gm}R_\gm/\gm R_\gm\\
    &\cong \Omega_{f,x}\otimes_{\Ocal_{X,x}}\Ocal_{\bar{X},\bar{x}}\\
    &\cong \Omega_{\bar{f},\bar{x}} =0
\end{align*}
Hence, by Nakayama's lemma, $\Omega_{f,x}=0$.
\end{proof}

\section{Upper bound on the size of preperiodic orbits: Proof of Theorem~\ref{theorem:singlemap}}
\label{section:bound1}

Proposition~\ref{proposition:key_fixed}, stated below is the key to proving our main theorems, which shows fixed points can be infinitesimally lifted.

\begin{proposition}
\label{proposition:key_fixed}
Let $f$ be an endomorphism of a separated scheme $X$ of finite type over a Noetherian domain $R$.
Let $\gm$ be a maximal ideal of $R$.
Suppose that $f_{R/\gm}$ has only unramified fixed $R/\gm$-points.
Suppose $x\in X(R)$ is an eventually $f$-fixed point. Then
\[
f(x)=x \iff f_{R/\gm}(\bar{x})=\bar{x}
\]
where $\bar{x}$ is $x \bmod \gm$.
\end{proposition}
\begin{proof}
The only if direction is clear.
For the converse, suppose $f_{R/\gm}(\bar{x})=\bar{x}$.
For an integer $r\geq1$, let $\pi_r : \Spec(R/\gm^r)\to\Spec(R)$ be the canonical map induced by projection of rings.
Since $x$ is a fixed point modulo $\gm$,
\[
x\circ \pi_1 = f \circ x \circ \pi_1.
\]
Since $x$ is an eventually $f$-fixed point, for each positive integer $r$, there exists an integer $\ell_r \geq 0$ such that
\[
f^{\ell_r}\circ x \circ \pi_r=f^{\ell_r + 1} \circ x \circ \pi_r.
\]
Hence we have the following two commutative diagrams
where each top and bottom arrows are equal,
$$
\begin{tikzcd}[row sep=huge, column sep=6.5em]
    \Spec(R/\gm)\arrow[d] \arrow[r, "x \circ \pi_1"]& X  \arrow[d, "f^{\ell_2}"]\\
    \Spec(R/\gm^2) \arrow[r, "f^{\ell_2} \circ x \circ \pi_2"'] \arrow[ru, "x \circ \pi_2" description]& X
\end{tikzcd}, \begin{tikzcd}[row sep=huge, column sep=6.5em]
    \Spec(R/\gm)\arrow[d] \arrow[r, "f \circ x \circ \pi_1"]& X  \arrow[d, "f^{\ell_2}"]\\
    \Spec(R/\gm^2) \arrow[r, "f^{\ell_2+1} \circ x \circ \pi_2"'] \arrow[ru, "f \circ x \circ \pi_2" description]& X
\end{tikzcd}.$$
For every $r\geq 1$, $\Spec R/\gm^r$ is a one point space, by abuse of notation we shall denote the point of $\Spec R/\gm^r$ by $s$ for every $r$.
Consider the corresponding diagrams on the stalks,
$$
\begin{tikzcd}[row sep=huge, column sep=6.5em]
    \Ocal_{\Spec(R/\gm), s} & \Ocal_{X,\bar{x}} \arrow[l, "(x \circ \pi_1)^\#"'] \arrow[ld, "(x \circ \pi_2)^\#" description] \\
    \Ocal_{\Spec(R/\gm^2),s}\arrow[u]  & \Ocal_{X,\bar{x}} \arrow[u, swap, "(f^{\ell_2})^\#" ] \arrow[l, "(f^{\ell_2} \circ x \circ \pi_2)^\#"]
\end{tikzcd}, \begin{tikzcd}[row sep=huge, column sep=6.5em]
    \Ocal_{\Spec(R/\gm), s} & \Ocal_{X,\bar{x}} \arrow[l, "(\f \circ x \circ \pi_1)^\#"'] \arrow[ld, "(f \circ x \circ \pi_2)^\#" description] \\
    \Ocal_{\Spec(R/\gm^2),s}\arrow[u]  & \Ocal_{X,\bar{x}} \arrow[u, swap, "(f^{\ell_2})^\#" ] \arrow[l, "(f^{\ell_2+1} \circ x \circ \pi_2)^\#"]
\end{tikzcd}.$$
The left hand side map of the above two diagrams is
\[
R_\gm/(\gm\otimes_R R_\gm)^2 \to R_\gm/(\gm\otimes_R R_\gm),
\]
which is a first order thickening, hence, by Lemma~\ref{lemma:basechange} and Proposition~\ref{proposition:lifting}, the diagonal map is unique implying the diagonal map of the two diagrams to be equal.
Thus $x \circ \pi_2$ and $f \circ x \circ \pi_2$ induce the same homomorphism on all stalks.
Hence, they are equal as scheme morphisms,
\[
x \circ \pi_2 = f \circ x \circ \pi_2,
\]
that is, $x$ is a fixed point modulo $\gm^2$.
Consider a similar commutative diagrams as above
$$
\begin{tikzcd}[row sep=huge, column sep=6.5em]
    \Spec(R/\gm^r)\arrow[d] \arrow[r, "x \circ \pi_r"]& X  \arrow[d, "f^{\ell_{r+1}}" ]\\
    \Spec(R/\gm^{r+1}) \arrow[r, "f^{\ell_{r+1}} \circ x \circ \pi_{r+1}"'] \arrow[ru, "x \circ \pi_{r+1}" description]& X
\end{tikzcd}, \begin{tikzcd}[row sep=huge, column sep=6.5em]
    \Spec(R/\gm^r)\arrow[d] \arrow[r, "f \circ x \circ \pi_r"]& X  \arrow[d, "f^{\ell_{r+1}}" ]\\
    \Spec(R/\gm^{r+1}) \arrow[r, "f^{\ell_{r+1}+1} \circ x \circ \pi_{r+1}"'] \arrow[ru, " f\circ x \circ \pi_{r+1}" description]& X
\end{tikzcd}$$
and by induction, we see
\[
  x \circ \pi_r = f \circ x \circ\pi_r
\]
for all $r\geq 1$. Thus, $x$ and $f \circ x$ are equal in
$X((R,\gm)\, \widehat{}\ )$, where $(R,\gm)\, \widehat{}$ is the completion of $R$ by $\gm$.
$X$ is separated and $R$ is Noetherian domain, so
$X(R)\to X((R,\gm)\, \widehat{}\ )$ is injective,
therefore, $x$ is $f$-fixed point in $X(R)$.
\end{proof}

We now restate Whang's theorem \cite[Theorem~1.2]{Whang} which plays a key role in proving our main theorems. Note that we have adjusted the theorem to match with our notations.
\begin{theorem}
\label{junhowhang}
\cite[Theorem~1.2]{Whang}
    Let $X$ be a separated scheme of finite presentation over a finitely presented domain $R$ of characteristic zero. There is an effective universal constant $C(X,R)$ such that, for any set $S$ of endomorphisms of $X/R$, and any $S$-periodic point $x\in X(R)$, we have
    $$|\Ocal_S(x)|\leq C(X,R).$$
\end{theorem}

\vspace{3mm}

\begin{proof}[Proof of Theorem~\ref{theorem:singlemap}]
Let $x\in X(R)$ be an $f$-preperiodic point. Let $n\geq 0$ and $m>0$ be the smallest integers that satisfy $f^n(x)=f^{n+m}(x)$.
Let $C=C(X,R)$ be the universal upper bound on the size of periodic orbits in Theorem~\ref{junhowhang}.
Then note that $m\leq C$. Let us denote $x$ modulo $\gm$ as $\bar{x}$, which is the composition of $\Spec(R/\gm)\to\Spec(R)$ and $x:\Spec(R) \to X$.
Suppose we have integers $0\leq i<j$ such that $f_{R/\gm}^i(\bar{x})=f_{R/\gm}^j(\bar{x})$.
Note that $x$ is eventually $f^{m(j-i)}$-fixed point and $f_{R/\gm}^i(\bar{x})$ is fixed point of $f_{R/\gm}^{m(j-i)}$, hence, by Proposition~\ref{proposition:key_fixed}, 
\[
f^i(x)=f^{m(j-i)}(f^i(x))=f^{i+m(j-i)}(x).
\]
By the minimality of $n$, we have $i\geq n$. We have shown that if $f_{R/\gm}^i(\bar{x})=f_{R/\gm}^j(\bar{x})$, for $i<j$, then $i\geq n$. That is, the following projection modulo $\gm$ map is a bijection, 
\[
\{f^i(x) : 0\leq i<n\}\to\{f_{R/\gm}^i(\bar{x}) : 0\leq i<n\}.
\]
Thus
\begin{align*}
|\Orbit_f(x)|&=n+m\\
& = |\{f_{R/\gm}^i(\bar{x}) : 0\leq i<n\}|+m\\
&\leq |X(R/\gm)|+C.
\end{align*}
Note that $R/\gm$ is a finite field. Indeed, $\ZZ$ is Jacobson ring and $R$ is finitely generated $\ZZ$-algebra then, by nullstellensatz, $\gm\cap\ZZ$ is maximal ideal of $\ZZ$ and $R/\gm$ is a finite extension of $\ZZ/(\gm\cap\ZZ)$.
\end{proof}

Let $f:X\to X$ be an endomorphism over $R$, and let $\gm\leq R$ be an ideal.
Suppose that $R$, $X$, $f$, and $\gm$ all satisfy the hypotheses of Proposition~\ref{proposition:key_fixed}.
Let $x\in X(R)$ be an $f$-preperiodic point so that $f^n(x)=f^{n+m}(x)$ with $n\geq0$ and $m>0$ minimal. Observing the proof of Theorem~\ref{theorem:singlemap}, we have shown that the following modulo $\gm$ map is a bijection, 
\begin{equation}
\label{eq:bijection}
    \{f^i(x) : 0\leq i<n\}\to\{f_{R/\gm}^i(\bar{x}) : 0\leq i<n\}.
\end{equation}

Thus, regardless of the size $m$ of the periodic part, we can bound the number of iterations required to enter a periodic cycle by $|X(R/\gm)|$.
So, if $x\in X(k)$ is a $f$-preperiodic point with periodic part of size $m$ then $|\Orbit_f(x)|\leq |X(R/\gm)|+m$.\\
Moreover, if sharper bounds on the size of periodic orbits are available, then improved bounds for preperiodic orbits can be derived, as demonstrated in the following Theorem~\ref{theorem:singlemap2} and Example~\ref{example:ZZ}.

\begin{theorem}
\label{theorem:singlemap2}
Let $(R, \gm)$ be a discrete valuation ring with finite residue field $k$ of characteristic $p>0$.
Let $X$ be a separated scheme of finite presentation over $R$ with regular special fiber.
Let $f:X\to X$ be an endomorphism such that $f_{k}$ has only unramified periodic $k$-points.
Then
\[
|\Orbit_f(x)| < |X(k)|((|k|^d-1)p^{v(p)}+1)
\]
for all $f$-preperiodic $x\in X(R)$, where $d=\dim X_k$ and $v$ is the normalized valuation on $R$.
\end{theorem}
\begin{proof}
    By \cite[Proposition~1]{Fakhruddin} and \cite[Proposition~2]{Fakhruddin}, all periodic orbits are bounded by $|X(k)|(|k|^d-1)p^{v(p)}$. Hence, by \eqref{eq:bijection}, we have
    \begin{align*}
        |\Orbit_f(x)|&<|X(k)|+|X(k)|(|k|^d-1)p^{v(p)}\\
        &=|X(k)|((|k|^d-1)p^{v(p)}+1)
    \end{align*}
    for all $f$-preperiodic point $x\in X(R)$.
\end{proof}
\begin{remark}
\begin{parts}
\Part{(a)} If we take $d=\max\{ dim_{\kappa(x)}m_x/m_x^2 : x\in X(k) \}$ where $m_x$ is the maximal ideal of $\Ocal_{X,x}$, then we can remove the hypothesis that $X$ has regular special fiber.
\Part{(b)} Theorem~\ref{theorem:singlemap2} applies to cases where Theorem~\ref{theorem:singlemap} does not, such as when $R$ is the ring of $p$-adic integers $\ZZ_p$.
\end{parts}
\end{remark}
\begin{example}
\label{example:ZZ}
Let $f:\ZZ^n\to\ZZ^n$ be a polynomial with Jacobian determinant $J$.
If there is a prime $p$ such that $J(y)\not\equiv 0\mod p$ for all periodic points $y\in\ZZ^n$ then for all preperiodic point $x\in\ZZ^n$, we have the following:
\begin{parts}
\Part{(a)}For $n=2$, in \cite{PezdaZ2} it is shown that the optimal bound for the size of periodic cycles is $24$, hence
\[
|\Orbit_f(x)|< p^2 + 24.
\]
\Part{(b)}For $n=3$, in \cite{Pezda3} it is shown that the optimal bound for the size of periodic cycles is $112$, hence
\[
|\Orbit_f(x)|< p^3 + 112.
\]
\Part{(c)}In general, in \cite{PezdaZN} it is shown that the periodic cycle of an $n$ variable integer polynomial is bounded by $2\cdot(4^n - 2^n)$, hence 
\[
|\Orbit_f(x)|< p^n + 2(4^n-2^n) .
\]
\end{parts}
\end{example}

\begin{remark}
Let $f:X\to X$ be an endomorphism and suppose there exists a subscheme $V\subseteq X$ such that $f(V)\subseteq V$.
Then we can obtain the upper bound on the size of preperiodic orbits as in Theorem~\ref{theorem:singlemap} and Theorem~\ref{theorem:singlemap2} only assuming that the restricted morphism $(f|_V)_{R/m}$ has only unramified periodic $R/m$-points.
\end{remark}

\section{Upper bound on the size of finite orbits: Proof of Theorem~\ref{theorem:manymap}}
\label{section:manymap}

In this section we prove Theorem~\ref{theorem:manymap}. The key step in its proof is Theorem~\ref{theorem:pathbound}, which relies on Lemma~\ref{lemma:check_periodic} and Proposition~\ref{prop:periodicequiv}. We then prove several auxiliary lemmas (Lemma~\ref{lemma:lvlincrease}, Lemma~\ref{lemma:lvl2}, Lemma~\ref{lemma:lvl3}), which, together with Theorem~\ref{theorem:pathbound}, allows us to complete the proof of Theorem~\ref{theorem:manymap}.

\begin{lemma}
\label{lemma:check_periodic}
Let $(X, S)$ be a dynamical system.
Suppose there is an integer constant $C$ such that $|\Orbit_S(x)|< C$, for all $S$-periodic point $x\in X$.
Then for any $x\in X$, we have
\[
x\text{ is }S\text{-periodic}\iff f^{C!}\circ y = y \text{ for all }f\in S \text{ and for all }y\in \Orbit_S(x)
\]
\end{lemma}
\begin{proof}
Suppose $x$ is $S$-periodic then every $y\in \Orbit_S(x)$ is $S$-periodic and hence $f^{C!}\circ y=y$ for all $f\in S$.\\
Conversely, suppose $f^{C!}\circ y=y$ for all $f\in S$ and for all $y\in \Orbit_S(x)$.
Fix $f\in S$, we need to show that $f|_{\Orbit_S(x)}:\Orbit_S(x)\to\Orbit_S(x)$ is a bijection.
For $y_1, y_2\in \Orbit_S(x)$, if $f\circ y_1=f\circ y_2$ then $y_1=f^{C!-1}\circ f\circ y_1=f^{C!-1}\circ f\circ y_2=y_2$.
For $y\in \Orbit_S(x)$, we have $f\circ (f^{C!-1} \circ y)=y$. Thus $f$ acts by permutation on $\Orbit_S(x)$.
As $f$ was chosen arbitrarily, $S$ acts by permutation on $\Orbit_S(x)$ and thus $x$ is $S$-periodic.
\end{proof}

\begin{proposition}
\label{prop:periodicequiv}
Let $X$ be a separated scheme of finite presentation over a finitely presented domain $R$ of characteristic zero and $\gm\leq R$ be a maximal ideal. For any collection $S\subseteq \{f\in\End(X/R): \Omega_{f_{R/\gm}}=0\}$ and for any $S$-finite point $x\in X(R)$, we have
\[
x \text{ is $S$-periodic} \iff x \bmod \gm\text{ is $S$-periodic}.
\]
\end{proposition}
\begin{proof}
Let us denote $x$ mod $\gm$ as $\bar{x}$, which is the composition of $\Spec(R/\gm)\to\Spec(R)$ and $x:\Spec(R)\to X$.
If $x$ is $S$-periodic then clearly $\bar{x}$ is $S$-periodic. Conversely, suppose $\bar{x}$ is $S$-periodic.
Let $C=C(X,R)$ be the universal upper bound on the size of periodic orbits in Theorem~\ref{junhowhang}.
Fix $y\in \Orbit_S(x)$ and fix $f\in S$.
Since $x$ is $S$ finite, we have $|\Orbit_f(y)|\leq |\Orbit_S(x)|<\infty$, so $y$ is $f$-preperiodic.
Then $y$ is eventually $f^{C!}$-fixed point.
Since $\bar{x}$ is $S$-periodic, any iteration of $\bar{x}$ is also $S$-periodic.
Hence, $\bar{y}:= (y \bmod \gm)$ is $S$-periodic and we have $f^{C!}\circ \bar{y}=\bar{y}$. Thus, by Proposition~\ref{proposition:key_fixed}, $y$ is $f^{C!}$-fixed point. Since $y$ and $f$ was chosen arbitrarily, by Lemma~\ref{lemma:check_periodic}, we conclude that $x$ is $S$-periodic.
\end{proof}

\begin{definition}
Let $(X,S)$ be a dynamical system.
For $x\in X$, consider the following collection of subsets of $\Orbit_S(x)$, 
\[
\Pcal_S(x):=\{ P\subseteq \Orbit_S(x) : \forall y_1, y_2 \in P, \exists w\in\langle S\rangle \text{ s.t. }w(y_1)=y_2 \text{ or }w(y_2)=y_1 \}.
\]
We shall call the elements $P$ of $\Pcal_S(x)$ a \emph{path of $\Orbit_S(x)$} or simply a $\emph{path}$.
\end{definition}

\begin{theorem}
\label{theorem:pathbound}
Let $X$ be a separated scheme of finite presentation over a finitely presented domain $R$ of characteristic zero. Let $\gm \leq R$ be a maximal ideal.
For any subset
\[
S\subseteq \{ f\in \End(X/R) : \Omega_{f_{R/\gm}}=0\},
\]
possibly infinite, there exists an effective constant $C=C(X,R)$ such that,
\[
|P|\leq C|X(R/\gm)|
\]
for all $S$-finite $x\in X(R)$ and for all $P\in \Pcal=\Pcal_S(x)$.
\end{theorem}
\begin{proof}
Let $x\in X(R)$ be a $S$-finite point and let $P\in \Pcal=\Pcal_S(x)$.
Let 
\[
\pi:P\to \bar{P}:=\{ (y \bmod \gm)\in X(R/\gm) : y\in P \}
\]
be the projection map modulo $\gm$.
For $\alpha\in \bar{P}$, define $S_\alpha:= \{ w\in \langle S \rangle \setminus id : w(\alpha)=\alpha\bmod \gm\}$ to be the collection of maps in $\langle S \rangle$ that fixes $\alpha$ modulo $\gm$, where $id$ is the identity map. Let $C=C(X,R)$ be the universal upper bound on the size of periodic orbits in Theorem~\ref{junhowhang}.\\
\textbf{Step 1}: We will show that if $S_\alpha=\emptyset$ then $|\pi^{-1}(\alpha)|=1$.\\
Suppose we have two distinct $y_1,y_2\in \pi^{-1}(\alpha)$, then there exists $w\in \langle S\rangle$ such that $w(y_1)=y_2$ or $w(y_2)=y_1$. Hence $w\in S_\alpha\neq\emptyset$.\\
\textbf{Step 2}:
We will show that if $S_\alpha\neq\emptyset$ then each $y\in \pi^{-1}(\alpha)$ is $S_\alpha$-periodic.\\
Fix $y\in \pi^{-1}(\alpha)$. 
Let $\f\in S_\alpha$ and $z\in \Orbit_{S_\alpha}(y)$.
Then $|\Orbit_\f(z)|\leq |\Orbit_S(x)|<\infty$, so $z$ is $\f$-preperiodic.
Since $z$ mod $\gm$ is fixed point of $\f$, by Proposition~\ref{prop:periodicequiv}, it follows that $z$ is $\f$-periodic and hence $\f^{C!}\circ z=z$.
Since $\f\in S_\alpha$ and $z\in \Orbit_{S_\alpha}(y)$ was chosen arbitrarily, by Lemma~\ref{lemma:check_periodic}, $y$ is $S_\alpha$-periodic.\\
\textbf{Step 3:}
We will show that if $S_\alpha\neq\emptyset$ then $\pi^{-1}(\alpha)$ is a single $S_\alpha$-periodic orbit.\\
For $y_1, y_2\in \pi^{-1}(\alpha)$, there exists $w\in \langle S\rangle$ such that $w(y_1)=y_2$ or $w(y_2)=y_1$.
Assume $w(y_1)=y_2$, then $w\in S_\alpha$ hence $\Orbit_{S_\alpha}(y_2)\subseteq \Orbit_{S_\alpha}(y_1)$.
Since $w$ acts by permutation on $\Orbit_{S_\alpha}(y_2)$, there exists $z\in\Orbit_{S_{\alpha}}(y_2)$ such that $w(z)=y_2$. $w$ also acts by permutation on $\Orbit_{S_\alpha}(y_1)$ and since $w(y_1)=y_2$, we have $y_1=z\in\Orbit_{S_\alpha}(y_2)$. Thus $\Orbit_{S_\alpha}(y_1)=\Orbit_{S_\alpha}(y_2)=\pi^{-1}(\alpha)$.\\
\textbf{Step 4:}
We will prove our desired upper bound on the size of paths.\\
By Step 1 and Step 3, it follows that
\[
|\pi^{-1}(\alpha)|\leq C.
\]
Thus we obtain our bound,
\[
|P| = \sum_{\alpha\in\bar{P}}|\pi^{-1}(\alpha)|\leq C|\bar{P}|\leq C|X(R/\gm)|.
\]
\end{proof}

We will now define some notations and prove several auxiliary lemmas (Lemma~\ref{lemma:lvlincrease}, Lemma~\ref{lemma:lvl2}, and Lemma~\ref{lemma:lvl3}) required to prove our main theorem, Theorem~\ref{theorem:manymap}.

\begin{definition}Let $(X,S)$ be a dynamical system and let $x\in X$. Let $\Pcal=\Pcal_S(x)$.
\begin{parts}
\Part{(i)}
For each $y\in \Orbit_S(x)$ define
\[
\Pcal_y := \{ P\in \Pcal : y\in P\}.
\]
\Part{(ii)}
For $y_1, y_2 \in \Orbit_S(x)$, define
\[
\Pcal_{y_1,y_2}:=\{ P\in \Pcal_{y_1}\cap \Pcal_{y_2} : \forall y\in P, \exists w_1, w_2\in \langle S \rangle \text{s.t. }w_1(y_1)=y, w_2(y)=y_2\}.
\]
\end{parts}
\end{definition}

\begin{definition}
Let $(X,S)$ be a dynamical system and let $x\in X$ be a $S$-finite point. Let $\Pcal=\Pcal_S(x)$. Given $y\in \Orbit_S(x)$, the set $\Pcal_{x,y}$ is partially ordered by inclusion. For each $y\in\Orbit_S(x)$, we shall define the following number
\[
\mathrm{lvl}(y):= \max(|P|: P \text{ is maximal in }\Pcal_{x,y})
\]
and we will call it the \emph{level of $y$}.
\end{definition}

\begin{lemma}
\label{lemma:lvlincrease}
Let $(X,S)$ be a dynamical system and let $x\in X$ be an $S$-finite point. For $P\in \Pcal=\Pcal_S(x)$ we have
\[
\mathrm{lvl}(y)\leq \mathrm{lvl}(w(y))
\]
for all $y\in P$ and for all $w\in \langle S \rangle$.
\end{lemma}
\begin{proof}
Let $Q\in \Pcal_{x,y}$ such that $|Q|=\mathrm{lvl}(y)$. For any $Q'\in \Pcal_{y, w(y)}$ note that $Q\cup Q'\in \Pcal_{x, w(y)}$ and we have
\[
\mathrm{lvl}(y)=|Q|\leq |Q\cup Q'|\leq \mathrm{lvl}(w(y)).
\]
\end{proof}

\begin{lemma}
\label{lemma:lvl2}
Let $(X,S)$ be a dynamical system and an $S$-finite point $x\in X$. Let $y\in \Orbit_S(x)$. For each $P\in \Pcal_{y}$  we have
\[
\{z\in P : \mathrm{lvl}(z)=\mathrm{lvl}(y)\}\in \Pcal_{y,y}.
\]
\end{lemma}
\begin{proof}
Fix $P\in \Pcal_y$. Note that any subset of a path $P$ is a path and $y\in \{z\in P : \mathrm{lvl}(z)=\mathrm{lvl}(y)\}$, so $\{z\in P : \mathrm{lvl}(z)=\mathrm{lvl}(y)\}\in \Pcal_y$.
Let $z\in P$ such that $\mathrm{lvl}(z)=\mathrm{lvl}(y)$. 
Let $Q_1\in \Pcal_{x, y}$ such that $|Q_1|=\mathrm{lvl}(y)$ and let $Q_2\in \Pcal_{x, z}$ such that $|Q_2|=\mathrm{lvl}(z)$.
There exist $w_1\in \langle S \rangle$ such that $w_1(y)=z$ or $w_1(z)=y$. Assume $w_1(y)=z$. Then we have
\[
|Q_1|\leq |Q_1\cup \{z\}|\leq |Q_2|=|Q_1|.
\]
Hence $z\in Q_1$, that is there is $w_2\in \langle S\rangle$ such that $w_2(z)=y$. Assume $w_1(z)=y$ then similarly,
\[
|Q_2|\leq |Q_2\cup \{y\}|\leq |Q_1|=|Q_2|.
\]
Hence $y\in Q_2$ that is there is $w_3\in \langle S\rangle$ such that $w_3(y)=z$. Therefore, 
\[
\{z\in P : \mathrm{lvl}(z)=\mathrm{lvl}(y)\}\in \Pcal_{y,y}.
\]
\end{proof}

\begin{lemma}
\label{lemma:lvl3}
Let $(X, S)$ be a dynamical system and an $S$-finite point $x\in X$. Let $y\in \Orbit_S(x)$. Then $\Pcal_{y,y}$ is closed under arbitrary union, that is for any collection $\{P_i\}_{i\in I} \subseteq P_{y,y}$ we have
\[
\bigcup_{i\in I} P_i \in \Pcal_{y,y}.
\]
\end{lemma}
\begin{proof}
Let $z_1,z_2\in \bigcup_{i\in I} P_i$. Assume, $z_1\in P_i$ and $z_2\in P_j$ for some $i,j\in I$. There exist $w_1, w_2\in \langle S \rangle$ such that $w_1(z_1)=y$ and $w_2(y)=z_2$. Hence $(w_2\circ w_1)(z_1)=z_2$, and it follows that $\bigcup_{i\in I} P_i$ is a path in $\Pcal_y$. Then it is clear that it is an element of $\Pcal_{y,y}$.
\end{proof}

\vspace{3mm}

\begin{proof}[Proof of Theorem~\ref{theorem:manymap}]
Let $C=C(X,R)$ be the universal upper bound on the size of periodic orbits in Theorem~\ref{junhowhang}.
Let $B:= C|X(R/\gm)|$ be upper bound on the size of paths in Theorem~\ref{theorem:pathbound}.
For an integer $t\geq 1$, define the following which is the set of all level $t$ elements in $\Orbit_S(x)$
\[
L(t)=\{ z\in \Orbit_S(x) : \mathrm{lvl}(z)=t \}.
\]
By Theorem~\ref{theorem:pathbound}, for $t>B$ we have $L(t)=\emptyset$.
Hence
\begin{equation}
\label{equation:partition}
    \Orbit_S(x) = \bigcup_{t=1}^B L(t).
\end{equation}
We will give an upper bound of $|L(t)|$ for each $t$, and hence obtain our upper bound for $|\Orbit_S(x)|$.\\
\textbf{Step 1:} We will show that $|L(1)|\leq 1$.\\
If $\mathrm{lvl}(x)>1$ then by Lemma~\ref{lemma:lvlincrease}, $L(1)=\emptyset$. If $\mathrm{lvl}(x)=1$ then, for any $P\in \Pcal_x$, by Lemma~\ref{lemma:lvl2}, $P\cap L(1)\in\Pcal_{x,x}$. Hence $|P\cap L(1)|\leq \mathrm{lvl}(x)=1$ and it follows that $P\cap L(1)=\{x\}$. Thus 
\[
L(1) = \bigcup_{P\in \Pcal_x} P\cap L(1) = \{x\}
\]
and $|L(1)|=1$.\\
\textbf{Step 2:} In this step, we will show that $|L(t)|\leq B|S|\sum^{t-1}_{i=1}|L(i)|$, for $t\geq 2$. We proceed by considering two cases: when at least one of the sets $L(1), \cdots, L(t-1)$ is nonempty, and when all of them are empty.\\
\textbf{Step 2-1:}  Suppose that all $L(1), \cdots, L(t-1)$ are empty sets.
If $\mathrm{lvl}(x)>t$ then, by Lemma~\ref{lemma:lvlincrease}, $|L(t)|=0$, so assume that $\mathrm{lvl}(x)=t$.
By Lemma~\ref{lemma:lvl2}, we see that the following is a path
\[
L(t) = \bigcup_{P\in \Pcal_{x}} \{z \in P : \mathrm{lvl}(z)=t\}\in \Pcal_{x,x}.
\]
Hence, $L(t)$ is a path, and by Theorem~\ref{theorem:pathbound}, $|L(t)|\leq B$.\\
\textbf{Step 2-2:}
Suppose at least one of $L(1), \cdots, L(t-1)$ is nonempty.
Then we can write our set $L(t)$ as follows
\begin{equation}
\label{equation:1}
    L(t) = \bigcup_{i=1}^{t-1}\bigcup_{\substack{y\in \Orbit_S(x)\\ \mathrm{lvl}(y)=i}}\bigcup_{\substack{f\in S \\ \mathrm{lvl}(f(y))=t}} \{z\in\Orbit_S(f(y)) : \mathrm{lvl}(z)=t\},
\end{equation}
the second union is taken over all $y\in \Orbit_S(x)$ such that $\mathrm{lvl}(y)=i$ and the third union is taken over all $f\in S$ such that $\mathrm{lvl}(f(y))=t$. 
By Lemma~\ref{lemma:lvl2} and Lemma~\ref{lemma:lvl3}, we see that the following is a path
\[
\{z\in \Orbit_S(f(y)) : \mathrm{lvl}(z)=t\} = \bigcup_{P\in \Pcal_{f(y)}} \{z \in P : \mathrm{lvl}(z)=t\}\in \Pcal_{f(y),f(y)}.
\]
Then by Theorem~\ref{theorem:pathbound}, we have
\[
|\{z\in \Orbit_S(f(y)) : \mathrm{lvl}(z)=t\}|\leq B.
\]
From (\ref{equation:1}), we obtain the following inequality
\begin{align*}
L(t) &\leq \sum_{i=1}^{t-1}\sum_{\substack{y\in \Orbit_S(x)\\ \mathrm{lvl}(y)=i}}\sum_{\substack{f\in S \\ \mathrm{lvl}(f(y))=t}} | \{z\in\Orbit_S(f(y)) : \mathrm{lvl}(z)=t\} |\\
&\leq \sum_{i=1}^{t-1} L(i)|S|B.
\end{align*}
Combining with \textbf{Step 2-1} we have
\[
|L(t)|\leq \max( B|S|\sum^{t-1}_{i=1}|L(i)|, B)\leq B|S|\sum^{t-1}_{i=1}|L(i)|.
\]
\textbf{Step 3:} We will show that there is an upper bound on the size of $|\Orbit_S(x)|$.\\
For $t\geq 2$, by \textbf{Step 2}, we have
\[
\sum^{t}_{i=1}|L(t)|-\sum^{t-1}_{i=1}|L(t)| =|L(t)|\leq B|S|\sum^{t-1}_{i=1}|L(i)|
\]
and
\[
\sum^{t}_{i=1}|L(t)|\leq (B|S|+1)\sum^{t-1}_{i=1}|L(t)|.
\]
Thus,
\[
\sum^{t}_{i=1}|L(t)|\leq (B|S|+1)^{t-1}|L(1)|\leq (B|S|+1)^{t-1}.
\]
Finally, we obtain our desired bound for $|\Orbit_S(x)|$,
\begin{align*}
|\Orbit_S(x)| &= \bigcup_{t=1}^B |L(t)|\\
&=\sum^B_{t=1}|L(t)|\\
&\leq (B|S|+1)^{B-1}\\
&=(C|X(R/\gm)||S|+1)^{C|X(R/\gm)|-1}.
\end{align*}
\end{proof}

\begin{example}
\label{example:affine unramified}
Let $X=\AA^n_\ZZ=\Spec(\ZZ[x_1, \ldots, x_n])$. Consider $f:X\to X$, induced by ring homomorphisms of the form
\begin{align*}
    \ZZ[x_1, \ldots, x_n] &\to \ZZ[x_1, \ldots, x_n]\\
    x_i &\mapsto x_i + p\cdot g_i(x_1, \ldots, x_n) + h_i(x_1^p, \ldots, x_n^p)
\end{align*}
where $g_i, h_i \in \ZZ[x_1, \ldots, x_n]$.
Then $f_{\FF_p}$ is unramified. Hence, we can apply Theorem~\ref{theorem:manymap} to a finite collection of such maps.
\end{example}

\begin{example}
For an integer matrix $A=(a_{ij})\in M_{n}(\ZZ)$, associate a monomial map $f_A :\GG^n_m \to \GG^n_m$ defined by
\[
f_A(x_1, \cdots, x_n) = (x^{a_{11}}_1x^{a_{12}}_2\cdots x^{a_{1n}}_n, \ldots, x^{a_{n1}}_1x^{a_{n2}}_2\cdots x^{a_{nn}}_n).
\]
Let $n$, $N$ be positive integers and let $p$ be a prime integer such that $p\nmid N$. Let
\[
S\subseteq \{f_A : A\in M_n(\ZZ) \text{ such that } p\nmid \det(A) \}
\] be a finite subset.
Note that $(f_A)_{\FF_p}$ is unramified if and only if $\det(A)\not\equiv 0 \bmod p$.
By Theorem~\ref{theorem:manymap}, there exists a constant $C=C(n, N)$ such that,
\[
|\Orbit_S(x)|< (C\cdot(p-1)^n|S|+1)^{C\cdot(p-1)^n-1},
\]
for all $S$-finite $x\in \GG^n_m(\ZZ[1/N])$. 
\end{example}

For a dynamical system $(X,S)$, we have bounded the size of finite $S$-orbits in Theorem~\ref{theorem:singlemap}, Theorem~\ref{theorem:singlemap2} and Theorem~\ref{theorem:manymap}.
In certain cases, we can bound the total number of $S$-finite points. We shall denote the set of all $S$-finite points by
\[
\mathrm{Fin}(X,S):=\{ x\in X : x \text{ is $S$-finite}\}=\bigcup_{\substack{x\in X \\ x \text{ is $S$-finite}}}\Orbit_S(x).
\]

\begin{proposition}
\label{proposition:projective}
Let $X$ be an irreducible normal projective variety over a number field $K$.
Let $S\subseteq \End(X/K)$ be any collection of endomorphisms of $X$ such that, $S$ contains at least one map $f$ with the following properties\\
(i) $\Omega_f=0$\\
(ii) $|\Per(X(K),f)|<\infty$\\
Then we have
\[
|\mathrm{Fin}(X(K), S)|\leq |\Per(X(K),f)|\cdot \deg(f)^{|\Xcal(\kappa(\gp))|},
\]
where $\Xcal$ is a projective model of $X$ and $\kappa(\gp)$ is the residue field of some prime $\gp$ of $K$.
\end{proposition}
\begin{proof}
Since $f$ is unramified and quasi-compact, it is quasi-finite.
As $f$ is proper and quasi-finite, it follows that $f$ is finite.
Finite morphisms preserve dimensions, that is $\dim X=\dim f(X)$, since $X$ is irreducible, $f$ is dominant.
For a dominant finite morphism, we can define the degree $\deg(f)=[K(X):f^*K(X)]<\infty$.
Let $\Ocal_K$ be the ring of integers of $K$.
By spreading out (see Lemma~\ref{lemma:spreading}), we have a model $\Xcal$ over $\Spec(\Ocal_K[1/\ell])$ of $X$ and an unramified morphism $\tilde{f}:\Xcal\to\Xcal$ that extends $f$.
Since unramified morphisms are stable under base change, for any prime $\gp$ of $\Ocal_K[1/\ell]$, $\tilde{f}_{\kappa(\gp)}$ is unramified.
Note that $\Xcal(\Ocal_K[1/\ell])=X(K)$ because $\Xcal$ is proper.
Recall that if a point is $f$-preperiodic, then it becomes periodic after at most $|\Xcal(\kappa(\gp))|$ number of iterations by $f$, see the paragraph after the proof of Theorem~\ref{theorem:singlemap}.
Since there are at most $\deg(f)$ number of points in a preimage of $f$, we can bound the number of preperiodic points of $f$ by
\[
|\Per(X(K),f)|\deg(f)^{|\Xcal(\kappa(f))|}.
\]
If $x\in \mathrm{Fin}(X(K), S)$ then $x\in \mathrm{Fin}(X(K), f)$.
Therefore,
\[
|\mathrm{Fin}(X(K), S)|\leq |\mathrm{Fin}(X(K), f)| \leq |\Per(X(K),f)|\deg(f)^{|\Xcal(\kappa(f))|}.
\]
\end{proof}

\section{Decidability of points with finite orbit: Proof of Theorem \ref{decidable}}
\label{section:decide}
In this section, we assume that the data defining a dynamical system $(X/R, S)$
is given explicitly.
In particular, $R$ is a finitely presented algebra over $\ZZ$ where the generators and relations are specified.
For the scheme $X$ over $R$, we assume that it is given by a finite affine open covering  $\{U_i=\Spec(A_i)\}$ where $A_i$ is explicitly given by a quotient of a polynomial ring with coefficients in $R$ and the relation for the quotient is given explicitly and the gluing morphisms $U_i\cap U_j\to U_i$ are explicitly given by $R$-algebra homomorphisms.
An endomorphism $f\in S$, is given by $f|_{U_{ij}}: U_{ij} \to V_i$
where $\{V_i=\Spec B_i\}_i$ and $\{U_{ij}=\Spec C_{ij}\}_{i,j}$ are finite affine coverings of $X$
such that the $R$-algebra homomorphisms $B_i\to C_{ij}$ corresponding to $f|_{U_{ij}}$ is explicitly given.

\begin{proof}[Proof of Theorem \ref{decidable}]
Let $X$ be an algebraic variety over $\Qbar$ and let $S$ be a finite
collection of unramified endomorphisms of $X$.
Given a point $x\in X(\Qbar)$, by spreading out, we can effectively
determine a finitely presented domain $R$ and a flat model $\Xcal/R$ of $X$ and
explicit $R$-morphisms $\tilde{f}:\Xcal\to\Xcal$ for each $f\in S$ satisfying the following:
\begin{parts}
\Part{(i)} $x$ descends to a point in $X(R)$
\Part{(ii)} $\tilde{f}_\Qbar=f$
\Part{(iii)} $\tilde{f}$ is unramified (by Lemma~\ref{lemma:spreading})
\end{parts}
We shall denote the collection of such $\tilde{f}$ by $\Scal$.
Let $\gm$ be any maximal ideal of $R$.
Since unramifiedness is preserved under base change, $\tilde{f}_{R/\gm}$ is unramified, for all $\tilde{f}\in \Scal$.
By Theorem~\ref{theorem:manymap}, there is an effective constant $C=C(\Xcal, R, \gm)$ such that, if $x$ is $\Scal$-preperiodic then $|\Orbit_\Scal(x)|\leq C$ .
Consider the ascending chain of finite sets $O_1\subseteq O_2 \subseteq \cdots \subseteq O_n \subseteq \cdots$ where $O_1=\{x\}$ and 
\[
O_{n+1}=O_n \cup \{ \tilde{f}(y) : y\in O_n, \ \tilde{f}\in \Scal \}
\]
for $n\geq 1$. Therefore, if $|O_{\lceil C \rceil}|\leq C$ then $x$ is $S$-finite point and otherwise $x$ is not $S$-finite, where ${\lceil C \rceil}$ is the smallest integer greater than or equal to $C$.
\end{proof}

In the following proposition, we show an easier way for effectively computing the bound for eventually fixed orbits of polynomial maps $f:\ZZ^n\to \ZZ^n$ such that $(Sing(f)\cap Fix(f))(\Qbar) = \emptyset$, that is the singular locus and the fixed locus of $f$ do not have any common $\Qbar$-points.

\begin{proposition}
Let $f:\ZZ^n \to \ZZ^n$ be a polynomial map.
Suppose the Jacobian determinant $J(x)$ of $f$ does not vanish on all of fixed $\Qbar$-points of $f$. 
Then there exists an effectively computable prime number $p$ such that
\[
|\Orbit_f(x)|\leq p^n
\]
for all eventually fixed point $x\in \ZZ^n$.
\end{proposition}
\begin{proof}
Let $f=(f_1, \cdots, f_n)$ and let $g_i(t)=f_i(t)-t_i$, for each $i$.
Since $J$ and $g_1, \cdots, g_n$ have no common zeroes, by Nullstellensatz, we have 
\begin{equation}
\label{null}
    h_0(t)J(t)+h_1(t)g_1(t)+\cdots+h_n(t)g_n(t)=k,
\end{equation}
for some $h_i\in \QQ[t_1,\cdots, t_n]$ and nonzero $k\in \QQ$.
By chasing denominators, we can multiply (\ref{null}) by an integer $N$ such that
\[
N(h_0(t)J(t)+h_1(t)g_1(t)+\cdots+h_n(t)g_n(t))=Nk\in \ZZ.
\]
Choose a prime $p$ such that $p\nmid Nk$.
Then for a fixed point $\alpha$ of $f$, we have 
\[
Nh_0(\alpha)J(\alpha)\equiv Nk \not\equiv 0 \mod p.
\]
Thus, the Jacobian determinant does not vanish on the fixed locus modulo $p$.
Therefore, $\bar{f} \bmod p$ has only unramified fixed $\FF_p$-points.
Thus by Theorem~\ref{theorem:singlemap}, we have $|\Orbit_f(x)|\leq p^n$, for all eventually fixed point $x\in \ZZ^n$.
\end{proof}

\end{document}